\pdfoutput=1
\RequirePackage{ifpdf}
\ifpdf 
\documentclass[pdftex]{sigma}
\else
\documentclass{sigma}
\fi

\def\BI{\mathfrak{BI}}

\def\N{\mathbb N}
\def\F{\mathbb F}
\def\Z{\mathbb Z}

\numberwithin{equation}{section}

\newtheorem{Theorem}{Theorem}[section]
\newtheorem{Corollary}[Theorem]{Corollary}
\newtheorem{Lemma}[Theorem]{Lemma}
\newtheorem{Proposition}[Theorem]{Proposition}
 { \theoremstyle{definition}
\newtheorem{Definition}[Theorem]{Definition}
\newtheorem{Example}[Theorem]{Example} }

\begin{document}

\allowdisplaybreaks

\newcommand{\arXivNumber}{1910.11446}

\renewcommand{\PaperNumber}{018}

\FirstPageHeading

\ShortArticleName{Finite-Dimensional Irreducible Modules of the Racah Algebra at Characteristic Zero}

\ArticleName{Finite-Dimensional Irreducible Modules\\ of the Racah Algebra at Characteristic Zero}

\Author{Hau-Wen HUANG~$^\dag$ and Sarah BOCKTING-CONRAD~$^\ddag$}

\AuthorNameForHeading{H.-W.~Huang and S.~Bockting-Conrad}

\Address{$^\dag$~Department of Mathematics, National Central University, Chung-Li 32001, Taiwan}
\EmailD{\href{mailto:hauwenh@math.ncu.edu.tw}{hauwenh@math.ncu.edu.tw}}

\Address{$^\ddag$~Department of Mathematical Sciences, DePaul University, Chicago, Illinois, USA}
\EmailD{\href{mailto:sarah.bockting@depaul.edu}{sarah.bockting@depaul.edu}}

\ArticleDates{Received November 12, 2019, in final form March 16, 2020; Published online March 24, 2020}

\Abstract{Assume that $\F$ is an algebraically closed field with characteristic zero. The Racah algebra~$\Re$ is the unital associative $\F$-algebra defined by generators and relations in the following way. The generators are $A$, $B$, $C$, $D$ and the relations assert that $[A,B]=[B,C]=[C,A]=2D$
and that each of $[A,D]+AC-BA$, $[B,D]+BA-CB$, $[C,D]+CB-AC$ is central in $\Re$. In this paper we discuss the finite-dimensional irreducible $\Re$-modules in detail and classify them up to isomorphism. To do this, we apply an infinite-dimensional $\Re$-module and its universal property. We additionally give the necessary and sufficient conditions for $A$, $B$, $C$ to be diagonalizable on finite-dimensional irreducible $\Re$-modules.}


\Classification{81R10; 16S37}

\section{Introduction}

Throughout this paper, we adopt the following conventions. Let $\F$ denote an algebraically closed field and let $\operatorname{char}\F$ denote the characteristic of~$\F$. Let $\Z$ denote the set of integers and let $\N$ denote the set of nonnegative integers. The bracket $[\,,\,]$ stands for the commutator.

In this paper we consider the Racah algebra $\Re$ over~$\F$ defined by generators and relations as follows. The generators are $A$, $B$, $C$, $D$ and the relations assert that
 \begin{gather*}
 [A,B]=[B,C]=[C,A]=2D
\end{gather*}
 and that each of
\begin{gather*}
[A,D]+AC-BA, \qquad [B,D]+BA-CB, \qquad [C,D]+CB-AC
\end{gather*}
is central in $\Re$. The Racah algebra $\Re$ is a universal analog of the original Racah algebras which first appeared in~\cite{Levy1965}. In that paper, the original Racah algebras were used to establish a link between representation theory and the quantum mechanical coupling of three angular momenta. Since that time, the connections between the Racah algebras and many other areas have been explored. We mention a few of them here. Their connections with the additive double-affine Hecke algebra of type $\big(C_1^\vee,C_1\big)$, the Bannai--Ito algebra, and the Lie algebras $\mathfrak{su}(2)$, $\mathfrak{su}(1,1)$ were investigated in \cite{gvz2013,R&BI2015,zhedanov1988,Huang:R<BImodules}. Their realizations via the Racah polynomials, the intermediate Casimir operators, and the superintegrable models in two dimensions were presented in \cite{Galbert,R&LD2014,gvz2014,R&BI2015,quadratic1991}. For information concerning the higher rank Racah algebras, see~\cite{HR2017,HR2019}.

We now mention an error in the literature on Racah algebras.
In \cite{Rmodule2019}, the authors considered the finite-dimensional irreducible modules of the original Racah algebras when $\operatorname{char}\F=0$. In \cite[Lemma~5.6]{Rmodule2019}, it was claimed that the defining generators can be diagonalized on any finite-dimensional irreducible module of the Racah algebras. This result was then used in their classification of finite-dimensional irreducible modules of the Racah algebras in \cite[Section~6]{Rmodule2019}. It turns out that \cite[Lemma~5.6]{Rmodule2019} is conditional. We give the following example to help illustrate the issue arising in~\cite{Rmodule2019}.

\begin{Example}\label{exam:R}
It is routine to verify that there exists a five-dimensional $\Re$-module $V$ that has an $\F$-basis $\{v_i\}_{i=0}^4$ with respect to which the matrices representing $A$, $B$, $C$, $D$ are $\frac{1}{4}$ times
\begin{gather*}
\begin{pmatrix}
15 &0 &0 &0 &0
\\
4 &3 &0 &0 &0
\\
0 &4 &-1 &0 &0
\\
0 &0 &4 &3 &0
\\
0 &0 &0 &4 &15
\end{pmatrix},
\qquad
\begin{pmatrix}
15 &-36 &0 &0 &0
\\
0 &3 &-6 &0 &0
\\
0 &0 &-1 &-6 &0
\\
0 &0 &0 &3 &-36
\\
0 &0 &0 &0 &15
\end{pmatrix},
\\
\begin{pmatrix}
-9 &36 &0 &0 &0
\\
-4 &15 &6 &0 &0
\\
0 &-4 &23 &6 &0
\\
0 &0 &-4 &15 &36
\\
0 &0 &0 &-4 &-9
\end{pmatrix},
\qquad
\begin{pmatrix}
18 &-54 &0 &0 &0
\\
6 &-15 &-3 &0 &0
\\
0 &2 &0 &3 &0
\\
0 &0 &-2 &15 &54
\\
0 &0 &0 &-6 &-18
\end{pmatrix},
\end{gather*}
respectively. It follows that
\begin{gather*}
[A,D]+AC-BA=[B,D]+BA-CB=[C,D]+CB-AC=0
\end{gather*}
on the $\Re$-module $V$. For each of $A$, $B$, $C$, it is straightforward to verify that its minimal polynomial on~$V$ is
\begin{gather*}
\left(
x+\frac{1}{4}
\right)
\left(
x-\frac{3}{4}
\right)^2
\left(
x-\frac{15}{4}
\right)^2.
\end{gather*}
Therefore none of $A$, $B$, $C$ is diagonalizable on~$V$. We now show that $V$ is in fact irreducible. Let~$W$ denote a nonzero $\Re$-submodule of~$V$. We will show that $W=V$. Observe that the element~$B$ has exactly three eigenvalues on~$V$, namely $\frac{15}{4}$, $\frac{3}{4}$, $-\frac{1}{4}$. A direct calculation yields that the corresponding eigenspaces are each of dimension~$1$ and are spanned by
\begin{gather}\label{counterexample}
v_0,
\qquad
3 v_0+v_1,
\qquad
27 v_0+12 v_1+ 8v_2,
\end{gather}
respectively. Since $W$ is nonzero, at least one of $\frac{15}{4}$, $\frac{3}{4}$, $-\frac{1}{4}$ is an eigenvalue of $B$ on $W$. Therefore~$W$ contains at least one of the elements listed in~(\ref{counterexample}). Observe that the $\Re$-module $V$ is generated by $v_0$. Thus, if $v_0\in W$ then $W=V$. If $3 v_0+v_1\in W$ then
\begin{gather*}
v_0=
\frac{2}{9}
\left(
A-2D
-\frac{9}{4}
\right)
(3v_0+v_1)\in W
\end{gather*}
and hence $W=V$. If $27 v_0+12 v_1+8v_2\in W$ then
\begin{gather*}
3 v_0+v_1
=-\frac{1}{18}
\left(
A+2D
-\frac{11}{4}
\right)
(27 v_0+12 v_1+ 8v_2)
\in W
\end{gather*}
and hence $W=V$. Therefore $W=V$. Since the $\Re$-module $V$ is irreducible, we now have a~counterexample to \cite[Lemma~5.6]{Rmodule2019}.
\end{Example}

In light of the above example, we see that the finite-dimensional irreducible $\Re$-modules are not yet completely classified. The goal of this paper is to provide such a classification. The idea of our classification comes from~\cite{Huang:2015}. We mention that a similar issue arises in the case of the Bannai--Ito algebra~$\BI$ \cite{BImodule2016} which is addressed by the first author in~\cite{Huang:BImodule}. The result \cite[Theorem~5.4]{Huang:R<BI} reveals that the Racah algebra $\Re$ is isomorphic to an $\F$-subalgebra of~$\BI$. As an application of~\cite{Huang:BImodule} and this result, the lattices of $\Re$-submodules of finite-dimensional irreducible $\BI$-modules are classified in~\cite{Huang:R<BImodules}.

The outline of this paper is as follows. In Section~\ref{s:classification} we state our classification of finite-dimensional irreducible $\Re$-modules in Theorem~\ref{thm:classification}. In Section~\ref{s:Verma} we display an infinite-dimen\-sio\-nal $\Re$-module and describe its universal property. In Section~\ref{s:irr} we give necessary and sufficient conditions for the irreducibility of finite-dimensional $\Re$-modules. In Section~\ref{s:iso} we study the isomorphism classes of finite-dimensional irreducible $\Re$-modules. In Section~\ref{s:proof} we give our proof of Theorem~\ref{thm:classification}.

\section{Statement of results}\label{s:classification}
In this section we more formally introduce the Racah algebra~$\Re$ and state the main result of the paper which gives a classification of the finite-dimensional irreducible modules of the Racah algebra~$\Re$. This main result will be proved later in Section~\ref{s:proof}.

\begin{Definition}[{\cite[Definition 3.1]{SH:2017-1}}] \label{defn:Racah}
The {\it Racah algebra} $\Re$ is the unital associative $\F$-algebra defined by generators and relations in the following way. The generators are $A$, $B$, $C$, $D$. The relations state that
\begin{gather}\label{r:D}
[A,B]=[B,C]=[C,A]=2D
\end{gather}
and that each of
\begin{gather*}
[A,D]+AC-BA, \qquad [B,D]+BA-CB, \qquad[C,D]+CB-AC
\end{gather*}
is central in $\Re$.
\end{Definition}
It follows from the above definition that the element $A+B+C$ is also central in $\Re$. For notational convenience, we let
\begin{gather}
\alpha = [A,D]+AC-BA, \label{alpha} \\
\beta = [B,D]+BA-CB,\label{beta}\\
\gamma = [C,D]+CB-AC,\label{gamma}\\
\delta = A+B+C. \label{delta}
\end{gather}

\begin{Lemma}\label{lem:delta}\quad
\begin{enumerate}\itemsep=0pt
\item[$(i)$] The Racah algebra $\Re$ is generated by the elements $A$, $B$, $C$.
\item[$(ii)$] The Racah algebra $\Re$ is generated by the elements $A$, $B$, $\delta$.
\end{enumerate}
\end{Lemma}
\begin{proof}
(i) By~(\ref{r:D}) the element $D$ can be expressed in terms of $A$, $B$. Hence (i) follows from Definition~\ref{defn:Racah}.
(ii) By (\ref{delta}) the element $C$ can be expressed in terms of $A$, $B$, $\delta$. Hence (ii) follows from~(i).
\end{proof}

\begin{Lemma}\label{lem:presentationR}The $\F$-algebra $\Re$ has a presentation given by generators $A$, $B$, $\alpha$, $\beta$, $\delta$ and relations
\begin{gather}
A^2 B-2 A B A+B A^2-2AB-2BA=2A^2-2A\delta +2\alpha,\label{AAB}\\
AB^2-2BAB+B^2 A-2 AB-2 BA= 2B^2-2B\delta-2\beta,\label{ABB}\\
\alpha A = A\alpha, \qquad\beta A = A\beta, \qquad\delta A = A\delta,\notag\\
\alpha B = B\alpha,\qquad\beta B = B\beta,\qquad\delta B = B\delta,\qquad
 \alpha \delta = \delta \alpha,\qquad \beta \delta = \delta \beta. \notag
\end{gather}
\end{Lemma}
\begin{proof}We know from Lemma~\ref{lem:delta}(ii) that $A$, $B$, $\alpha$, $\beta$, $\delta$ generate~$\Re$. Observe that $C=\delta-A-B$ by \eqref{delta} and $D=\frac{1}{2}[A,B]$ by~\eqref{r:D}. The result can now be obtained by either using these two facts to eliminate~$C$, $D$ from the presentation of $\Re$ given in Definition~\ref{defn:Racah} or by using $D=\frac{1}{2}[A,B]$ to eliminate~$D$ from the presentation of $\Re$ given in \cite[Proposition~3.4]{SH:2017-1}.
\end{proof}

In the following proposition, we assert the existence of certain finite-dimensional $\Re$-modules and describe the actions of the generators of $\Re$ on these modules. A reader familiar with the theory of tridiagonal pairs will immediately recognize the form of the matrices representing~$A$ and~$B$ as precisely those given in Terwilliger's 2001 seminal work on tridiagonal pairs \cite[Theorem~3.2]{lp2001}.

\begin{Proposition}\label{prop:Rd}For any $a,b,c\in \F$ and any $d\in \N$ there exists a $(d+1)$-dimensional $\Re$-module $R_d(a,b,c)$ satisfying each of the following conditions:
\begin{enumerate}\itemsep=0pt
\item[$(i)$] There exists an $\F$-basis $\{v_i\}_{i=0}^d$ for $R_d(a,b,c)$ with respect to which the matrices represen\-ting $A$ and $B$ are
\begin{gather*}
\begin{pmatrix}
\theta_0 & & & &{\bf 0}
\\
1 &\theta_1
\\
&1 &\theta_2
 \\
& &\ddots &\ddots
 \\
{\bf 0} & & &1 &\theta_d
\end{pmatrix},
\qquad
\begin{pmatrix}
\theta_0^* &\varphi_1 & & &{\bf 0}
\\
 &\theta_1^* &\varphi_2
\\
 & &\theta_2^* &\ddots
 \\
 & & &\ddots &\varphi_d
 \\
{\bf 0} & & & &\theta_d^*
\end{pmatrix},
\end{gather*}
respectively, where
\begin{gather*}
\theta_i=\big(a+\textstyle\frac{d}{2}-i\big) \big(a+\textstyle\frac{d}{2}-i+1\big), \qquad 0\leq i\leq d,\\
\theta^*_i=\big(b+\textstyle\frac{d}{2}-i\big)\big(b+\textstyle\frac{d}{2}-i+1\big), \qquad 0\leq i\leq d,
\\
\varphi_i=i(i-d-1)\big(a+b+c+\textstyle\frac{d}{2}-i+2\big)\big(a+b-c+\textstyle\frac{d}{2}-i+1\big),
\qquad 1\leq i\leq d.
\end{gather*}

\item[$(ii)$] The elements $\alpha$, $\beta$, $\delta$ act on $R_d (a,b,c)$ as scalar multiplication by
\begin{gather*}
(c-b)(c+b+1)\big(a-\textstyle\frac{d}{2}\big)\big(a+\textstyle\frac{d}{2}+1\big),\\
(a-c)(a+c+1)\big(b-\textstyle\frac{d}{2}\big)\big(b+\textstyle\frac{d}{2}+1\big),\\
\textstyle\frac{d}{2} \big(\textstyle\frac{d}{2}+1\big)+a(a+1)+b(b+1)+c(c+1),
\end{gather*}
respectively.
\end{enumerate}
\end{Proposition}
\begin{proof} Using Lemma \ref{lem:presentationR}, this result can be verified through routine, though tedious, computations.
\end{proof}

In order to state our main result more succinctly, we will use the following conventions and definitions.
Let $d\in \N$ and let $\mathbf{P}=\mathbf{P}_d$ denote the set of all $(a,b,c)\in \F^3$ that satisfy
\begin{gather*}
a+b+c+1,-a+b+c,a-b+c,a+b-c\not\in \big\{\tfrac{d}{2}-i\,\big|\,i=1,2,\ldots,d\big\}.
\end{gather*}
We define an action of the abelian group $\{\pm 1\}^3$ on $\mathbf{P}$ by
\begin{gather*}
(a,b,c)^{(-1,1,1)} = (-a-1,b,c),\qquad
(a,b,c)^{(1,-1,1)} = (a,-b-1,c),\\
(a,b,c)^{(1,1,-1)} = (a,b,-c-1)
\end{gather*}
for all $(a,b,c)\in \mathbf{P}$. We let $\mathbf{P}/\{\pm 1\}^3$ denote the set of the $\{\pm 1\}^3$-orbits of $\mathbf{P}$. For $(a,b,c)\in \mathbf{P}$, let $[a,b,c]$ denote the $\{\pm 1\}^3$-orbit of $\mathbf{P}$ that contains $(a,b,c)$. We are now ready to state the classification of finite-dimensional irreducible $\Re$-modules.

\begin{Theorem}\label{thm:classification} Assume that $\F$ is algebraically closed with $\operatorname{char}\F=0$. Let $d$ denote a nonnegative integer. Let $\mathbf{M}$ denote the set of all isomorphism classes of irreducible $\Re$-modules that have dimension $d+1$. Then there exists a bijection $\mathcal R\colon \mathbf{P}/\{\pm 1\}^3 \to \mathbf{M}$ given by
\begin{gather*}
[a,b,c]\mapsto \hbox{the isomorphism class of $R_d(a,b,c)$}
\end{gather*}
for all $[a,b,c]\in \mathbf{P}/\{\pm 1\}^3$.
\end{Theorem}

We will give a proof of Theorem~\ref{thm:classification} in Section~\ref{s:proof}.

\section[An infinite-dimensional $\Re$-module and its universal property]{An infinite-dimensional $\boldsymbol{\Re}$-module and its universal property}\label{s:Verma}

In this section we introduce an infinite-dimensional $\Re$-module and its universal property in order to prove Theorem~\ref{thm:classification}. For convenience the following conventions are used throughout the rest of this paper. We let $a$, $b$, $c$, $\nu$ denote any scalars in~$\F$. We define the following families of parameters associated with $a$, $b$, $c$, $\nu$:
\begin{gather}
\theta_i =\big(a+\textstyle\frac{\nu}{2}-i\big)\big(a+\textstyle\frac{\nu}{2}-i+1\big)\qquad \hbox{for all $i\in \Z$}, \label{theta_i}\\
\theta^*_i=\big(b+\textstyle\frac{\nu}{2}-i\big)\big(b+\textstyle\frac{\nu}{2}-i+1\big)\qquad\hbox{for all $i\in \Z$}, \label{theta_is}\\
\phi_i=i(i-\nu-1)\big(a-b+c-\textstyle\frac{\nu}{2}+i\big)\big(a-b-c-\textstyle\frac{\nu}{2}+i-1\big)\qquad
\hbox{for all $i\in \Z$},\label{phi_i}\\
\varphi_i=i(i-\nu-1)\big(a+b+c+\textstyle\frac{\nu}{2}-i+2\big)\big(a+b-c+\textstyle\frac{\nu}{2}-i+1\big)
\qquad\hbox{for all $i\in \Z$},\label{varphi_i}\\
\zeta=(c-b)(c+b+1)\big(a-\textstyle\frac{\nu}{2}\big)\big(a+\textstyle\frac{\nu}{2}+1\big),\label{omega}\\
\zeta^*=(a-c)(a+c+1)\big(b-\textstyle\frac{\nu}{2}\big)\big(b+\textstyle\frac{\nu}{2}+1\big),\label{omages}\\
\eta=\textstyle\frac{\nu}{2} \big(\textstyle\frac{\nu}{2}+1\big)+a(a+1)+b(b+1)+c(c+1).\label{eta}
\end{gather}

\begin{Proposition}\label{prop:Verma}There exists an $\Re$-module $M_\nu (a,b,c)$ satisfying each of the following conditions:
\begin{enumerate}\itemsep=0pt
\item[{\rm (i)}] There exists an $\F$-basis $\{m_i\}_{i=0}^\infty$ for $M_\nu (a,b,c)$ with respect to which the matrices representing $A$ and $B$ are
\begin{gather*}
\begin{pmatrix}
\theta_0 & & & & &{\bf 0}
\\
1 &\theta_1
\\
&1 &\theta_2
 \\
 & &\cdot &\cdot
 \\
& & &\cdot &\cdot
 \\
 {\bf 0} & & & &\cdot &\cdot
\end{pmatrix},
\qquad
\begin{pmatrix}
\theta_0^* &\varphi_1 & & & &{\bf 0}
\\
 &\theta_1^* &\varphi_2
\\
 & &\theta_2^* &\cdot
 \\
 & & &\cdot &\cdot
 \\
& & & &\cdot &\cdot
 \\
 {\bf 0} & & & & &\cdot
\end{pmatrix},
\end{gather*}
respectively.

\item[{\rm (ii)}] The elements $\alpha$, $\beta$, $\delta$ act on $M_\nu (a,b,c)$ as scalar multiplication by $\zeta$, $\zeta^*$, $\eta$, respectively.
\end{enumerate}
\end{Proposition}
\begin{proof}Using Lemma \ref{lem:presentationR}, this result can be verified through routine computations.
\end{proof}

Throughout the rest of this paper we will let $\{m_i\}_{i=0}^\infty$ denote the $\F$-basis for $M_\nu(a,b,c)$ from Proposition~\ref{prop:Verma}(i). The following result is an immediate consequence of Proposition~\ref{prop:Verma}(i).

\begin{Lemma}\label{lem:mi}$m_{j+1}=\prod\limits_{h=i}^j (A-\theta_h)m_i$ for any $i,j\in \N$ with $ i\leq j$.
\end{Lemma}

Shortly we will describe the $\Re$-module $M_\nu(a,b,c)$ in an alternate way. To aid us in this endeavor, we first recall a Poincar\'{e}--Birkhoff--Witt basis for $\Re$.

\begin{Lemma}[{\cite[Theorem 5.1]{SH:2017-1}}]\label{lem:basisURA}
The elements
\begin{gather*}
A^i D^j B^k \alpha^r \delta^s \beta^t\qquad \hbox{for all $i,j,k,r,s,t\in \N$}\label{eq:basisURA}
\end{gather*}
form an $\F$-basis of~$\Re$.
\end{Lemma}

Let $I_\nu(a,b,c)$ denote the left ideal of $\Re$ generated by the elements
\begin{gather}
B-\theta_0^*,\label{U1}\\
(B-\theta_1^*)(A-\theta_0)-\varphi_1,\label{U2}\\
\alpha-\zeta,\qquad\beta-\zeta^*,\qquad\delta-\eta.
\label{U3}
\end{gather}
We now consider certain cosets of $\Re/I_\nu(a,b,c)$.

\begin{Lemma}\label{lem:cosets}For each $n\in \N$, each of the following holds:
\begin{enumerate}\itemsep=0pt
\item[$(i)$] $BA^n+I_\nu(a,b,c)$ is an $\F$-linear combination of $A^i+I_\nu(a,b,c)$ for all $0\leq i\leq n$,
\item[$(ii)$] $DA^n+I_\nu(a,b,c)$ is an $\F$-linear combination of $A^i+I_\nu(a,b,c)$ for all $0\leq i\leq n+1$,
\item[$(iii)$] $D^n+I_\nu(a,b,c)$ is an $\F$-linear combination of $A^i+I_\nu(a,b,c)$ for all $0\leq i\leq n$.
\end{enumerate}
\end{Lemma}
\begin{proof} (i) We proceed by induction on $n$. Since $I_\nu(a,b,c)$ contains the element (\ref{U1}), the statement holds for $n=0$. Since $I_\nu(a,b,c)$ contains both of the elements~(\ref{U1}) and~(\ref{U2}), the statement holds for $n=1$. Now suppose $n\geq 2$. Right multiplying each side of (\ref{AAB}) by $A^{n-2}$ yields that
\begin{gather*}
A^2 B A^{n-2}
-2 A B A^{n-1}
+B A^n
-2 A B A^{n-2}
-2 B A^{n-1}
=2 A^n-2 A^{n-1}\delta+2A^{n-2}\alpha.
\end{gather*}
Since $I_\nu(a,b,c)$ contains each of the elements listed in~(\ref{U3}), it follows that $B A^n$ is congruent to
\begin{gather}\label{ABAn-1}
2 A B A^{n-1}+2 A B A^{n-2}-A^2 B A^{n-2}+2 B A^{n-1}+2 A^n-2\eta A^{n-1}+2\zeta A^{n-2}
\end{gather}
modulo $I_\nu(a,b,c)$. By the inductive hypothesis, the element (\ref{ABAn-1}) is congruent to an $\F$-linear combination of $A^i$, for all $0\leq i\leq n$, modulo $I_\nu(a,b,c)$. Therefore (i) follows.

(ii) Observe that $DA^n=\frac{1}{2}\big(ABA^{n}-BA^{n+1}\big)$ by (\ref{r:D}). In light of this fact, the result now follows from Lemma~\ref{lem:cosets}(i).

(iii) We proceed by induction on $n$. The statement holds trivially for $n=0$. Now suppose that $n\geq 1$. By the inductive hypothesis, $D^n=DD^{n-1}$ is congruent to an $\F$-linear combination of
\begin{gather}\label{DA^i+I}
DA^i, \qquad 0\leq i\leq n-1,
\end{gather}
modulo $I_\nu(a,b,c)$. By Lemma \ref{lem:cosets}(ii) each of the elements listed in~(\ref{DA^i+I}) is an $\F$-linear combination of~$A^k$, for all $0\leq k\leq n$, modulo $I_\nu(a,b,c)$. Therefore the result follows.
\end{proof}

\begin{Lemma}\label{lem:R/I_basis}The $\F$-vector space $\Re/I_\nu(a,b,c)$ is spanned by
\begin{gather*}
A^i+I_\nu(a,b,c) \qquad \hbox{for all $i\in \N$}.
\end{gather*}
\end{Lemma}
\begin{proof}By Lemma \ref{lem:basisURA}, the $\F$-vector space $\Re/I_\nu(a,b,c)$ is spanned by
\begin{gather}\label{coset1}
A^i D^j B^k \alpha^r \delta^s \beta^t+I_\nu(a,b,c)\qquad \hbox{for all $i,j,k,r,s,t\in \N$}.
\end{gather}
Since $I_\nu(a,b,c)$ contains the elements listed in (\ref{U1}) and (\ref{U3}), each of the elements listed in~(\ref{coset1}) can be expressed as an $\F$-linear combination of
\begin{gather*}
A^i D^j+I_\nu(a,b,c) \qquad \hbox{for all $i,j\in \N$}.
\end{gather*}
The result now follows from these facts along with Lemma~\ref{lem:cosets}(iii).
\end{proof}

We are now ready to give our second description of $M_\nu(a,b,c)$.

\begin{Theorem}\label{thm:M=R/I}There exists a unique $\Re$-module homomorphism
\begin{gather*}
\Phi\colon \ \Re/I_\nu(a,b,c)\to M_\nu(a,b,c)
\end{gather*}
that sends $1+I_\nu(a,b,c)$ to $m_0$. Moreover, $\Phi$ is an isomorphism.
\end{Theorem}
\begin{proof}Consider the $\Re$-module homomorphism $\Psi\colon \Re\to M_\nu(a,b,c)$ that sends $1$ to $m_0$. By Proposition \ref{prop:Verma}(i), the elements (\ref{U1}) and (\ref{U2}) are in the kernel of $\Psi$. By Proposition \ref{prop:Verma}(ii), the elements listed in (\ref{U3}) are also in the kernel of $\Psi$. Hence $I_\nu(a,b,c)$ is contained in the kernel of~$\Psi$. It follows that $\Psi$ induces an $\Re$-module homomorphism $\Phi\colon \Re/I_\nu(a,b,c) \to M_\nu(a,b,c)$ that maps $1+I_\nu(a,b,c)$ to~$m_0$. Observe that $\Phi$ is the unique $\Re$-module homomorphism with the desired property since $\Re/I_\nu(a,b,c)$ is generated by $1+I_\nu(a,b,c)$ as an $\Re$-module.

By Lemma \ref{lem:mi} the homomorphism $\Phi$ sends
\begin{gather}\label{coset3}
\prod_{h=1}^{i-1}(A-\theta_h)+I_\nu(a,b,c)
\end{gather}
to $m_i$ for all $i\in \N$. Since $\{m_i\}_{i=0}^\infty$ are linearly independent, the cosets~(\ref{coset3}) are linearly independent. Combining this with Lemma~\ref{lem:R/I_basis}, we see that the cosets~(\ref{coset3}) are an $\F$-basis for $\Re/I_\nu(a,b,c)$. Therefore $\Phi$ is an isomorphism.
\end{proof}

As a consequence of Theorem~\ref{thm:M=R/I}, the $\Re$-module $M_\nu(a,b,c)$ satisfies the following universal property.

\begin{Proposition}\label{prop:universal}If $V$ is an $\Re$-module which has a vector $v\in V$ satisfying
\begin{gather*}
Bv=\theta_0^* v,\\
(B-\theta_1^*)(A-\theta_0)v=\varphi_1 v,\\
\alpha v=\zeta v,\qquad\beta v=\zeta^* v,\qquad\delta v=\eta v,
\end{gather*}
then there exists a unique $\Re$-module homomorphism $M_\nu(a,b,c)\to V$ that sends~$m_0$ to~$v$.
\end{Proposition}

For the rest of the present paper, we will consider the case $\nu=d$. Define $N_d(a,b,c)$ to be the $A$-cyclic $\F$-subspace of $M_d(a,b,c)$ generated by the element $m_{d+1}$.

\begin{Lemma}\label{lem:N}$N_d(a,b,c)$ is an $\Re$-submodule of $M_d(a,b,c)$ with the $\F$-basis $\{m_i\}_{i=d+1}^\infty$.
\end{Lemma}
\begin{proof}Recall from Lemma \ref{lem:mi} that
\begin{gather*}
(A-\theta_i) m_i=m_{i+1}\qquad \hbox{for all $i\geq d+1$}.
\end{gather*}
It follows from this fact that $\{m_i\}_{i=d+1}^\infty$ is an $\F$-basis for $N_d(a,b,c)$.

We now show that $N_d(a,b,c)$ is an $\Re$-submodule of $M_d(a,b,c)$. By Proposition~\ref{prop:Verma}(i),
\begin{gather*}
(B-\theta_i^*)m_i=\varphi_i m_{i-1} \qquad\hbox{for all $i\geq d+1$}.
\end{gather*}
By (\ref{varphi_i}), the scalar $\varphi_{d+1}=0$ when $\nu=d$. Hence $N_d(a,b,c)$ is $B$-invariant. By Proposition~\ref{prop:Verma}(ii), the element $\delta$ acts on $N_d(a,b,c)$ as scalar multiplication by $\eta$.
It now follows from Lemma \ref{lem:delta}(ii) that
 $N_d(a,b,c)$ is an $\Re$-submodule of $M_d(a,b,c)$.
\end{proof}

Recall the $\Re$-module $R_d(a,b,c)$ from Proposition~\ref{prop:Rd}. In the sequel we display how the $\Re$-module $R_d(a,b,c)$ is connected to $M_\nu(a,b,c)$. For convenience we let $\{v_i\}_{i=0}^d$ denote the $\F$-basis for $R_d(a,b,c)$ from Proposition~\ref{prop:Rd}(i) in the rest of this paper.

\begin{Lemma}\label{lem:M/N}There exists a unique $\Re$-module isomorphism
\begin{gather*}
M_d(a,b,c)/N_d(a,b,c)\to R_d(a,b,c)
\end{gather*}
that sends $m_i+N_d(a,b,c)$ to $v_i$ for all $0\leq i\leq d$.
\end{Lemma}
\begin{proof}By Lemma \ref{lem:N}, $M_d(a,b,c)/N_d(a,b,c)$ is a $(d+1)$-dimensional $\Re$-module with the $\F$-basis
\begin{gather}\label{mi+N}
\{m_i+N_d(a,b,c)\}_{i=0}^d.
\end{gather}
Observe that the matrices representing $A$ and $B$ with respect to the $\F$-basis $\{v_i\}_{i=0}^d$ for $R_d(a,b,c)$ are identical with the matrices representing $A$ and $B$ with respect to the $\F$-basis (\ref{mi+N}) for $M_d(a,b,c)/N_d(a,b,c)$ by Propositions~\ref{prop:Rd}(i) and~\ref{prop:Verma}(i). By Propositions~\ref{prop:Rd}(ii) and~\ref{prop:Verma}(ii), the actions of $\delta$ on $R_d(a,b,c)$ and $M_d(a,b,c)/N_d(a,b,c)$ are scalar multiplication by the same scalar~$\eta$. In light of these comments, the result now follows from Lemma~\ref{lem:delta}(ii).
\end{proof}

\begin{Proposition}\label{prop:R}Suppose that $V$ is an $\Re$-module which has a vector $v\in V$ satisfying
\begin{gather}\label{md+1_ker}
\prod_{i=0}^d (A-\theta_i) v=0.
\end{gather}
If there is an $\Re$-module homomorphism $M_d(a,b,c)\to V$ that sends $m_0$ to~$v$, then there exists an $\Re$-module homomorphism $R_d(a,b,c)\to V$ that sends~$v_0$ to~$v$.
\end{Proposition}
\begin{proof}Let $\varrho$ denote the $\Re$-module homomorphism $M_d(a,b,c)\to V$ that sends~$m_0$ to~$v$. By Lemma~\ref{lem:mi}, we have
\begin{gather*}
m_{d+1}=\prod_{i=0}^d (A-\theta_i) m_0.
\end{gather*}
Combining this with~(\ref{md+1_ker}), we see that $m_{d+1}$ is in the kernel of $\varrho$. Therefore $N_d(a,b,c)$ is contained in the kernel of~$\varrho$. By Lemma~\ref{lem:N}, there exists an $\Re$-module homomorphism $M_d(a,b,c)/N_d(a,b,c)\to V$ that sends $m_0+N_d(a,b,c)$ to~$v$. The result follows from this fact along with Lemma~\ref{lem:M/N}.
\end{proof}

\section[Conditions for the irreducibility of $R_d(a,b,c)$]{Conditions for the irreducibility of $\boldsymbol{R_d(a,b,c)}$}\label{s:irr}

In this section, we derive the necessary and sufficient conditions for $R_d(a,b,c)$ to be irreducible in terms of the parameters $a$, $b$, $c$, $d$. Throughout this section, we let
\begin{gather}\label{e:wi}
w_i=\prod_{h=0}^{i-1} (A-\theta_{d-h}) v_0, \qquad 0\leq i\leq d.
\end{gather}

\begin{Lemma}\label{lem:irr1}If the $\Re$-module $R_d(a,b,c)$ is irreducible, then each of the following holds:
\begin{enumerate}\itemsep=0pt
\item[$(i)$] $\operatorname{char} \F=0$ or $\operatorname{char} \F>d$,
\item[$(ii)$] $a+b+c+1, a+b-c\not\in \big\{{\frac{d}{2}-i}\,\big|\,i=1,2,\ldots,d\big\}$.
\end{enumerate}
\end{Lemma}
\begin{proof} Suppose that there is an integer $i$, with $1\leq i\leq d$, such that $\varphi_i=0$. By Proposition~\ref{prop:Rd}, the $\F$-subspace $W$ of $R_d(a,b,c)$ spanned by $\{v_h\}_{h=i}^d$ is invariant under $A$, $B$, $\delta$. It follows from Lemma~\ref{lem:delta}(ii) that $W$ is an $\Re$-submodule of $R_d(a,b,c)$, a contradiction to the irreducibility of $R_d(a,b,c)$. Therefore $\varphi_i\not=0$ for all $1\leq i\leq d$, which is equivalent to~(i) and~(ii) by~(\ref{varphi_i}).
\end{proof}

\begin{Lemma}\label{lem:iso2} The elements $\{w_i\}_{i=0}^d$ form an $\F$-basis for $R_d(a,b,c)$.
\end{Lemma}
\begin{proof} It follows from Proposition \ref{prop:Rd}(i) that
\begin{gather*}
v_i=\prod_{h=0}^{i-1} (A-\theta_{h}) v_0, \qquad 0\leq i\leq d.
\end{gather*}
Comparing this with~(\ref{e:wi}), the result now follows.
\end{proof}

\begin{Proposition}\label{prop:iso2}The $\Re$-module $R_d(a,b,c)$ is isomorphic to the $\Re$-module $R_d(-a-1,b,c)$. Moreover, the matrices representing $A$ and $B$ with respect to the $\F$-basis $\{w_i\}_{i=0}^d$ for $R_d(a,b,c)$ are
\begin{gather}\label{AB_Rd(-a-1,b,c)}
\begin{pmatrix}
\theta_d & & & &{\bf 0}
\\
1 &\theta_{d-1}
\\
&1 &\theta_{d-2}
 \\
& &\ddots &\ddots
 \\
{\bf 0} & & &1 &\theta_0
\end{pmatrix},
\qquad
\begin{pmatrix}
\theta_0^* &\phi_1 & & &{\bf 0}
\\
 &\theta_1^* &\phi_2
\\
 & &\theta_2^* &\ddots
 \\
 & & &\ddots &\phi_d
 \\
{\bf 0} & & & &\theta_d^*
\end{pmatrix},
\end{gather}
respectively.
\end{Proposition}
\begin{proof}By Proposition~\ref{prop:Rd}(i), there exists an $\F$-basis $\{u_i\}_{i=0}^d$ for $R_d(-a-1,b,c)$ with respect to which the matrices representing $A$ and $B$ are equal to the matrices displayed in~(\ref{AB_Rd(-a-1,b,c)}). By Lemma~\ref{lem:iso2}, it suffices to show that there is an $\Re$-module homomorphism $R_d(a,b,c)\to R_d(-a-1,\allowbreak b,c)$ that sends $w_i$ to $u_i$ for all $0\leq i\leq d$.

Observe that $B u_0=\theta_0^* u_0$ and a direct calculation yields that
\begin{gather*}
(B-\theta_1^*)(A-\theta_0) u_0=\varphi_1 u_0.
\end{gather*}
By Proposition \ref{prop:Rd}(ii), the elements $\alpha$, $\beta$, $\delta$ act on $R_d(-a-1,b,c)$ as scalar multiplication by~$\zeta$,~$\zeta^*$,~$\eta$, respectively. According to Proposition~\ref{prop:universal}, there exists a unique $\Re$-module homomorphism $M_d(a,b,c)\to R_d(-a-1,b,c)$ that sends~$m_0$ to~$u_0$. By inspecting the matrix representing~$A$ given in~(\ref{AB_Rd(-a-1,b,c)}) we see that
\begin{gather*}
\prod_{i=0}^d (A-\theta_i)u_0=0.
\end{gather*}
Hence there exists a $\Re$-module homomorphism
\begin{gather*}
R_d(a,b,c)\to R_d(-a-1,b,c)
\end{gather*}
that maps $v_0$ to $u_0$ by Proposition~\ref{prop:R}. It now follows from~(\ref{e:wi}) that this homomorphism sends $w_i$ to $u_i$ for all $0\leq i\leq d$. The result follows.
\end{proof}

\begin{Lemma}\label{lem:irr2}If the $\Re$-module $R_d(a,b,c)$ is irreducible, then each of the following holds:
\begin{enumerate}\itemsep=0pt
\item[{\rm (i)}] $\operatorname{char} \F=0$ or $\operatorname{char} \F>d$,

\item[{\rm (ii)}] $a+b+c+1, -a+b+c, a-b+c, a+b-c\not\in \big\{ \frac{d}{2}-i \,\big|\,i=1,2,\ldots,d \big\}$.
\end{enumerate}
\end{Lemma}
\begin{proof}
By Proposition \ref{prop:iso2}, the $\Re$-module $R_d(a,b,c)$ is isomorphic to $R_d(-a-1,b,c)$. Hence the result follows by applying Lemma~\ref{lem:irr1} to both $R_d(a,b,c)$ and $R_d(-a-1,b,c)$.
\end{proof}

Shortly we will show that the converse of Lemma \ref{lem:irr2} is also true.
To aid us in doing so, we establish the following notation.
We define
\begin{gather*}
R=\prod_{h=1}^{d} (B-\theta^*_h ),\\
S_i=\prod_{h=1}^{d-i}(A-\theta_{d-h+1} ), \qquad 0\leq i\leq d.
\end{gather*}
It follows from Proposition~\ref{prop:Rd}(i) that $Rv$ is a scalar multiple of $v_0$ for all $v\in R_d(a,b,c)$. Thus, for any integers~$i$,~$j$ with $0\leq i,j\leq d$, there exists a unique $L_{ij}\in \F$ such that
\begin{gather}\label{defn:Lij}
RS_i v_j=L_{ij} v_0.
\end{gather}
By examining Proposition \ref{prop:Rd}(i) further, we see that
\begin{gather}
L_{ij}=0,\qquad 0\leq i<j\leq d,\label{lowertriangular}\\
L_{ij}=(\theta_{i}-\theta_{j-1}) L_{i,j-1}+L_{i-1,j-1}, \qquad 1\leq j\leq i \leq d.\label{L:rr}
\end{gather}
It follows from Proposition \ref{prop:iso2} that
\begin{gather}\label{Li0}
L_{i0}=\prod_{h=1}^{i} (\theta_0^*-\theta_{d-h+1}^*) \prod_{h=1}^{d-i} \phi_h, \qquad 0\leq i\leq d.
\end{gather}
Solving the recurrence relation (\ref{L:rr}) with the initial conditions~(\ref{lowertriangular}) and~(\ref{Li0}) yields that
\begin{gather}\label{Lij}
L_{ij}={d-i+j \choose j}{i \choose j}{d \choose j}^{-1}
\prod_{h=1}^{i-j} (\theta_0^*-\theta_{d-h+1}^*)
\prod_{h=1}^{d-i} \phi_h \prod_{h=1}^j \varphi_h, \qquad 0\leq j\leq i\leq d.
\end{gather}

\begin{Theorem}\label{thm:irrR}The $\Re$-module $R_d(a,b,c)$ is irreducible if and only if both of the following conditions hold:
\begin{enumerate}\itemsep=0pt
\item[$(i)$] $\operatorname{char} \F=0$ or $\operatorname{char} \F>d$,
\item[$(ii)$] $a+b+c+1, -a+b+c, a-b+c, a+b-c\not\in \big\{
 \frac{d}{2}-i \,\big|\,i=1,2,\ldots,d\big\}$.
\end{enumerate}
\end{Theorem}
\begin{proof}$(\Rightarrow)$ This is immediate from Lemma \ref{lem:irr2}.

$(\Leftarrow)$ To see the irreducibility of $R_d(a,b,c)$, we assume that $W$ is a nonzero $\Re$-submodule of $R_d(a,b,c)$ and show that $W=R_d(a,b,c)$. Pick a nonzero vector $w\in W$. Since $W$ is invariant under $A$ and $B$, it follows that
\begin{gather}\label{W}
R S_i w\in W, \qquad 0\leq i\leq d.
\end{gather}
Since $\{v_i\}_{i=0}^d$ is an $\F$-basis for $R_d(a,b,c)$, there exist $a_j\in \F$, for $0\leq j\leq d$, such that
\begin{gather*}
w=\sum_{j=0}^d a_j v_j.
\end{gather*}
It now follows from (\ref{defn:Lij}) that
\begin{gather}\label{RSw}
RS_i w=\left(\sum_{j=0}^d L_{ij} a_j\right) v_0, \qquad 0\leq i\leq d.
\end{gather}

Recall the parameters $\{\phi_i\}_{i\in \Z}$ and $\{\varphi_i\}_{i\in \Z}$ from~(\ref{phi_i}) and~(\ref{varphi_i}), respectively. It follows from our assumptions (i) and (ii) that the scalars $\varphi_i\not=0$ and $\phi_i\not=0$ for all $1\leq i\leq d$. Let $L$ denote the $(d+1)\times (d+1)$ matrix, indexed by $0,1,\ldots,d$, with $(i,j)$-entry given by $L_{ij}$ for all $0\leq i,j\leq d$.
By~(\ref{lowertriangular}), the square matrix $L$ is lower triangular. By~(\ref{Lij}), the diagonal entries of~$L$ are
\begin{gather*}
L_{ii}=\prod_{h=1}^{d-i} \phi_h \prod_{h=1}^{i}\varphi_i,\qquad 0\leq i\leq d,
\end{gather*}
which we know to be nonzero. Therefore the matrix $L$ is nonsingular. Since $w$ is nonzero at least one of $\{a_j\}_{j=0}^d$ is nonzero. Hence there exists an integer $i$ with $0\leq i\leq d$ such that
\begin{equation}
\sum_{j=0}^d L_{ij} a_j\not=0.\label{eq:Laneq0}
\end{equation}
Combining \eqref{eq:Laneq0} with (\ref{W}) and (\ref{RSw}), we find that $v_0\in W$. Since the $\Re$-module $R_d(a,b,c)$ is generated by $v_0$, it follows that $W=R_d(a,b,c)$ and so $R_d(a,b,c)$ is irreducible.
\end{proof}

\section[The isomorphism class of the $\Re$-module $R_d(a,b,c)$]{The isomorphism class of the $\boldsymbol{\Re}$-module $\boldsymbol{R_d(a,b,c)}$}\label{s:iso}

In Proposition \ref{prop:iso2}, we showed that the $\Re$-module $R_d(a,b,c)$ is isomorphic to the $\Re$-module $R_d(-a-1,b,c)$. In this section, we discuss the isomorphism class of $R_d(a,b,c)$ in further detail.

\begin{Proposition}\label{prop:iso1}The $\Re$-module $R_d(a,b,c)$ is isomorphic to the $\Re$-module $R_d(a,b,-c-1)$.
\end{Proposition}
\begin{proof}By Proposition \ref{prop:Rd}(i), there are $\F$-bases for $R_d(a,b,c)$ and $R_d(a,b,-c-1)$ with respect to which the matrices representing~$A$ and~$B$ are the same. By Proposition~\ref{prop:Rd}(ii), the actions of~$\delta$ on $R_d(a,b,c)$ and $R_d(a,b,-c-1)$ are both scalar multiplication by the same scalar~$\eta$. Hence $R_d(a,b,c)$ is isomorphic to $R_d(a,b,-c-1)$ by Lemma~\ref{lem:delta}(ii).
\end{proof}

\begin{Proposition}\label{prop:iso3}If the $\Re$-module $R_d(a,b,c)$ is irreducible, then $R_d(a,b,c)$ is isomorphic to the $\Re$-module $R_d(a,-b-1,c)$.
\end{Proposition}
\begin{proof}By Proposition \ref{prop:Rd}(i), there is an $\F$-basis $\{u_i\}_{i=0}^d$ for $R_d(a,-b-1,c)$ with respect to which the matrices representing~$A$ and~$B$ are
\begin{gather}\label{ABRd(a,-b-1,c)}
\begin{pmatrix}
\theta_0 & & & &{\bf 0}
\\
1 &\theta_1
\\
&1 &\theta_2
 \\
& &\ddots &\ddots
 \\
{\bf 0} & & &1 &\theta_d
\end{pmatrix},
\qquad
\begin{pmatrix}
\theta_d^* &\phi_d & & &{\bf 0}
\\
 &\theta_{d-1}^* &\phi_{d-1}
\\
 & &\theta_{d-2}^* &\ddots
 \\
 & & &\ddots &\phi_1
 \\
{\bf 0} & & & &\theta_0^*
\end{pmatrix},
\end{gather}
respectively. Since the $\Re$-module $R_d(a,b,c)$ is irreducible, it follows from Theorem \ref{thm:irrR} that $\phi_i\not=0$ for all $1\leq i\leq d$. Thus we may set
\begin{gather*}
v=\sum_{i=0}^d
\prod_{h=1}^i
\frac{\theta_0^*-\theta_{d-h+1}^*}{\phi_{d-h+1}} u_i.
\end{gather*}
A direct calculation yields that $Bv=\theta_0^* v$ and
\begin{gather*}
(B-\theta_1^*)(A-\theta_0) v
=\varphi_1 v.
\end{gather*}
By Proposition \ref{prop:Rd}(ii), the elements $\alpha$, $\beta$, $\delta$ act on $R_d(a,-b-1,c)$ as scalar multiplication by~$\zeta$,~$\zeta^*$,~$\eta$, respectively. According to Proposition~\ref{prop:universal}, there exists a unique $\Re$-module homomorphism $M_d(a,b,c)\to R_d(a,-b-1,c)$ that maps~$m_0$ to~$v$.
By inspecting the matrix representing~$A$ given in (\ref{ABRd(a,-b-1,c)}), we see that
\begin{gather*}
\prod_{i=0}^d (A-\theta_i)v=0.
\end{gather*}
Hence there exists an $\Re$-module homomorphism
\begin{gather}\label{R(b)->R(-b)}
R_d(a,b,c)\to R_d(a,-b-1,c)
\end{gather}
that sends $v_0$ to $v$ by Proposition \ref{prop:R}. Since the $\Re$-module $R_d(a,b,c)$ is irreducible, the $\Re$-module $R_d(a,-b-1,c)$ is also irreducible by Theorem~\ref{thm:irrR}. Therefore (\ref{R(b)->R(-b)}) is an isomorphism.
\end{proof}

We end this section with a simple combination of Propositions~\ref{prop:iso2},~\ref{prop:iso1}, and~\ref{prop:iso3}.

\begin{Theorem}\label{thm:iso}If the $\Re$-module $R_d(a,b,c)$ is irreducible, then $R_d(a,b,c)$ is isomorphic to each of the $\Re$-modules $R_d(-a-1,b,c)$, $R_d(a,-b-1,c)$ and $R_d(a,b,-c-1)$.
\end{Theorem}

\section{The proof of Theorem \ref{thm:classification}}\label{s:proof}

Theorems \ref{thm:irrR} and \ref{thm:iso} indicate that the map~$\mathcal R$ in Theorem~\ref{thm:classification} is well-defined. In this section, we shall show that $\mathcal R$ is a bijection.

\begin{Lemma}\label{lem:Schur}Assume that $\F$ is algebraically closed. If $V$ is a finite-dimensional irreducible $\Re$-module, then each central element of~$\Re$ acts on~$V$ as scalar multiplication.
\end{Lemma}
\begin{proof}This result follows from applying Schur's lemma to $\Re$.
\end{proof}

\begin{Lemma}\label{lem:theta}For any $i\in \Z$, each of the following hold:
\begin{enumerate}\itemsep=0pt
\item[$(i)$] $\theta_{i+1}+\theta_{i-1}=2(\theta_i+1)$,
\item[$(ii)$] $\theta_{i+1}\theta_{i-1}=\theta_i(\theta_i-2)$.
\end{enumerate}
\end{Lemma}
\begin{proof}The result can be routinely verified using (\ref{theta_i}).
\end{proof}

\begin{Theorem}\label{thm:surjective}Assume that $\F$ is algebraically closed with $\operatorname{char}\F=0$. Let $d$ denote a nonnegative integer. If $V$ is a $(d + 1)$-dimensional irreducible $\Re$-module, then there exist $a, b, c \in \F$ such that the $\Re$-module $R_d(a, b, c)$ is isomorphic to~$V$.
\end{Theorem}
\begin{proof}Given any scalar $\kappa \in \F$, we define
\begin{gather*}
\vartheta_i(\kappa)
=(\kappa-i)(\kappa-i+1)
\qquad
\hbox{for all $i\in \Z$}.
\end{gather*}
Since $\operatorname{char}\F=0$, for any distinct integers $i,j$, the scalars $\vartheta_i(\kappa)$ and $\vartheta_j(\kappa)$ are equal if and only if $i+j=2\kappa+1$. In particular $\{\vartheta_i(\kappa)\}_{i=0}^{-\infty}$ contains infinitely many values.

Since $\F$ is algebraically closed, we may choose a scalar $\kappa\in \F$ such that $\vartheta_0(\kappa)$ is an eigenvalue of~$A$ on~$V$. Since~$V$ is of dimension $d+1$, there are at most $d+1$ distinct eigenvalues of~$A$ on~$V$. Thus, there exists an integer $j\leq 0$ such that $\vartheta_j(\kappa)$ is an eigenvalue of~$A$ but $\vartheta_{j-1}(\kappa)$ is not an eigenvalue of~$A$ on~$V$. Set
\begin{gather*}
a=\kappa-j-\tfrac{d}{2}.
\end{gather*}
Similarly, there exists a scalar $\lambda\in \F$ and an integer $k\leq 0$ such that $\vartheta_k(\lambda)$ is an eigenvalue of~$B$ but $\vartheta_{k-1}(\lambda)$ is not an eigenvalue of~$B$ in~$V$. We set
\begin{gather*}
b=\lambda-k-\tfrac{d}{2}.
\end{gather*}
Observe that under these settings, we have
\begin{gather}
\theta_i = \vartheta_{i+j}(\kappa) \qquad\hbox{for all $i\in \Z$},\label{setting:a}\\
\theta_i^* = \vartheta_{i+k}(\lambda)\qquad\hbox{for all $i\in \Z$}.\label{setting:b}
\end{gather}
By Lemma \ref{lem:Schur}, the element $\delta$ acts on $V$ as scalar multiplication. Since $\F$ is algebraically closed, there exists a scalar $c\in \F$ such that the action of $\delta$ on $V$ is the scalar multiplication by
\begin{gather*}
\eta=\textstyle\frac{d}{2}\big(\frac{d}{2}+1\big)+a(a+1)+b(b+1)+c(c+1).
\end{gather*}
To prove the theorem, it now suffices to show that there exists an $\Re$-module isomorphism from $R_d(a,b,c)$ into~$V$.

Given any $T\in \Re$ and $\theta\in \F$, we let
\begin{gather*}
V_T(\theta)=\{v\in V\,|\, Tv=\theta v\}.
\end{gather*}
Pick any $v\in V_A(\theta_0)$. Applying each side of~(\ref{AAB}) to~$v$ and using Lemma~\ref{lem:theta} to simplify the result, we obtain that
\begin{gather}\label{AAB_V(theta)}
(A-\theta_{-1})(A-\theta_1) B v=2(\theta_0(\theta_0-\eta)+\alpha)v.
\end{gather}
Left multiplying each side of (\ref{AAB_V(theta)}) by $(A-\theta_0)$, we obtain that
\begin{gather*}
(A-\theta_{-1})(A-\theta_0)(A-\theta_1) B v=0.
\end{gather*}
By (\ref{setting:a}), the scalar $\theta_{-1}$ is not an eigenvalue of $A$ in $V$. Hence
\begin{gather*}
(A-\theta_0)(A-\theta_1) B v=0.
\end{gather*}
In other words $(A-\theta_1) B v\in V_A(\theta_0)$ and therefore $V_A(\theta_0)$ is invariant under $(A-\theta_1) B$. Since~$\F$ is algebraically closed, there exists an eigenvector $u$ of $(A-\theta_1) B$ in $V_A(\theta_0)$. Similarly, there exists an eigenvector $w$ of $(B-\theta_1^*)A$ in $V_B(\theta_0^*)$. Define
\begin{gather}
u_i=\prod_{h=0}^{i-1} (B-\theta_h^*) u \qquad \hbox{for all $i\in \N$},\label{vi}\\
w_i=\prod_{h=0}^{i-1} (A-\theta_h) w\qquad \hbox{for all $i\in \N$}. \label{wi}
\end{gather}

We now proceed by induction to show that
\begin{gather}\label{claim}
(A-\theta_i)u_i\in \operatorname{span}_\F\{u_0,u_1,\ldots,u_{i-1}\} \qquad \hbox{for all $i\in \N$}.
\end{gather}
Since $u$ is an eigenvector of $(A-\theta_1) B$ in $V_A(\theta_0)$, the claim is true for $i=0,1$.
Now suppose that $i\geq 2$. Applying each side of~(\ref{ABB}) to $u_{i-2}$, we obtain that
\begin{gather}\label{ABBvi-2}
\big(A B^2 -2 BAB +B^2 A - 2 A B-2 BA-2 B^2+2 \eta B\big) u_{i-2}=-2 \beta u_{i-2}.
\end{gather}
By Lemma \ref{lem:Schur}, the right-hand side of (\ref{ABBvi-2}) is a scalar multiple of~$u_{i-2}$. Using the inductive hypothesis, (\ref{vi}), and Lemma~\ref{lem:theta}(i), we find that the left-hand side of~(\ref{ABBvi-2}) is equal to
\begin{gather}\label{ABBvi-2:LH'}
(A-\theta_i) u_i
\end{gather}
plus an $\F$-linear combination of $u_0,u_1,\ldots,u_{i-1}$. Combining the above results, the claim~(\ref{claim}) follows.

Next, we show that $\{u_i\}_{i=0}^d$ is an $\F$-basis for $V$. Suppose on the contrary that there is an integer $h$ with $0\leq h\leq d-1$ such that $u_{h+1}$ is an $\F$-linear combination of $u_0,u_1,\ldots,u_h$. Let~$W$ denote the $\F$-subspace of~$V$ spanned by $u_0,u_1,\ldots,u_h$.
Since~$W$ is $B$-invariant by~(\ref{vi}) and $A$-invariant by~(\ref{claim}), it follows that $W$ is an $\Re$-submodule of $V$ by Lemma~\ref{lem:delta}(ii). Since~$V$ is irreducible, this forces that $W=V$.
By construction, $W$ is of dimension at most~$d$, a~contradiction. Therefore $\{u_i\}_{i=0}^d$ is an $\F$-basis for $V$.
By a similar argument, it follows that
\begin{gather}\label{claim'}
(B-\theta_i^*) w_i\in \operatorname{span}_\F\{w_0,w_1,\ldots,w_{i-1}\}\qquad \hbox{for all $i\in \N$}
\end{gather}
and $\{w_i\}_{i=0}^d$ is an $\F$-basis for $V$.

By (\ref{claim}), the matrix representing $A$ with respect to the $\F$-basis $\{u_i\}_{i=0}^d$ is upper triangular with diagonal entries $\{\theta_i\}_{i=0}^d$. By the Cayley--Hamilton theorem, we have
\begin{gather}\label{CH}
\prod\limits_{i=0}^d (A-\theta_i) w_0=0.
\end{gather}
In other words $Aw_d=\theta_d w_d$ by (\ref{wi}). Hence the matrix representing $A$ with respect to the $\F$-basis $\{w_i\}_{i=0}^d$ is
\begin{gather*}
\begin{pmatrix}
\theta_0 & & & &{\bf 0}
\\
1 &\theta_1
\\
&1 &\theta_2
 \\
& &\ddots &\ddots
 \\
{\bf 0} & & &1 &\theta_d
\end{pmatrix}.
\end{gather*}

By (\ref{claim'}), the matrix representing $B$ with respect to the $\F$-basis $\{w_i\}_{i=0}^d$ is upper triangular with diagonal entries $\{\theta_i^*\}_{i=0}^d$. We let $\{\varphi_i'\}_{i=1}^d$ denote its superdiagonal entries as follows
\begin{gather}\label{B}
\begin{pmatrix}
\theta_0^* &\varphi_1' & & &{\bf *}
\\
 &\theta_1^* &\varphi_2'
\\
 & &\theta_2^* &\ddots
 \\
 & & &\ddots &\varphi_d'
 \\
{\bf 0} & & & &\theta_d^*
\end{pmatrix}.
\end{gather}
Applying each side of (\ref{AAB}) to $w_{i-1}$, we obtain from the coefficients of $w_i$ that
\begin{gather}\label{varphi'}
\varphi_{i+1}'-2\varphi_i'+\varphi_{i-1}'
=(\theta_i-\theta_{i-1}+2)\theta_i^*
-(\theta_i-\theta_{i-1}-2)\theta_{i-1}^*
+2(\theta_i+\theta_{i-1}-\eta)
\end{gather}
for all $1\leq i\leq d$, where $\varphi_0'$ and $\varphi_{d+1}'$ are interpreted as zero.
It is straightforward to verify that $\{\varphi_i\}_{i=1}^d$ also satisfy the recurrence relation (\ref{varphi'}). Since $\operatorname{char}\F=0$, the corresponding homogeneous recurrence relation
\begin{gather*}
\sigma_{i+1}-2\sigma_i+\sigma_{i-1}=0, \qquad 1\leq i\leq d,
\end{gather*}
with the initial values $\sigma_0=0$ and $\sigma_{d+1}=0$, has the unique solution $\sigma_i=0$ for all $0\leq i\leq d+1$. Therefore $\varphi_i'=\varphi_i$ for all $1\leq i\leq d$.

Up to this point, we have shown that
\begin{gather}
Bw_0=\theta_0^* w_0,\label{even:U1}\\
(B-\theta_1^*)(A-\theta_0)w_0=\varphi_1 w_0,\label{even:U2}\\
\delta w_0=\eta w_0.\label{even:U3}
\end{gather}

Applying each side of (\ref{ABB}) to $w_0$ and using~(\ref{even:U2}) to simplify the resulting equation, we find that
\begin{gather}\label{even:U4}
\beta w_0=\zeta^* w_0.
\end{gather}
Using logic similar to what was used to show (\ref{even:U2}), we obtain $(A-\theta_1)(B-\theta_0^*)u_0=\varphi_1 u_0$. Now, by applying each side of (\ref{AAB}) to $u_0$ and using the above equation to simplify the resulting equation, we find that $\alpha u_0=\zeta u_0$. It follows from Lemma~\ref{lem:Schur} that
\begin{gather}\label{even:U5}
\alpha w_0=\zeta w_0.
\end{gather}
In view of (\ref{even:U1})--(\ref{even:U5}), it follows from Proposition~\ref{prop:universal} that there exists a unique $\Re$-module homomorphism $M_d(a,b,c)\to V$ that sends~$m_0$ to~$w_0$. By Proposition~\ref{prop:Verma}(i), the entries above the superdiagonal in~(\ref{B}) are zero.
Combining the above $\Re$-module homomorphism \mbox{$M_d(a,b,c)\to V$} with~(\ref{CH}), there is an $\Re$-module homomorphism
\begin{gather}\label{R->V}
R_d(a,b,c)\to V
\end{gather}
that sends $v_0$ to $w_0$ by Proposition~\ref{prop:R}. Since the $\Re$-module $V$ is irreducible, the homomorphism~(\ref{R->V}) is onto. Since $R_d(a,b,c)$ and $V$ are both of dimension $d+1$, it follows that (\ref{R->V}) is an isomorphism. The result follows.
\end{proof}

\begin{Lemma}\label{lem:trace}The traces of $A$, $B$, $C$ on the $\Re$-module $R_d(a,b,c)$ are equal to $d+1$ times
\begin{gather*}
a^2+a+\frac{d(d+2)}{12},
\qquad
b^2+b+\frac{d(d+2)}{12},
\qquad
c^2+c+\frac{d(d+2)}{12},
\end{gather*}
respectively.
\end{Lemma}
\begin{proof}To compute the traces of $A$, $B$ on $R_d(a,b,c)$, use Proposition~\ref{prop:Rd}(i). By~(\ref{delta}), the trace of $C$ is equal to the trace of $\delta$ minus the trace of $A+B$ on $R_d(a,b,c)$. Use the above facts along with Proposition~\ref{prop:Rd}(ii) to compute the trace of $C$ on $R_d(a,b,c)$.
\end{proof}

The following is a quick consequence of Theorems \ref{thm:iso}, \ref{thm:surjective} and Lemma~\ref{lem:trace}.

\begin{Corollary}\label{cor:inj}\looseness=-1 Assume that $\F$ is algebraically closed with $\operatorname{char}\F=0$. Let $V$ denote a $(d + 1)$-dimensional irreducible $\Re$-module. Let $\operatorname{tr} A$, $\operatorname{tr} B$, $\operatorname{tr} C$ denote the traces of $A$, $B$, $C$ on~$V$, respectively. Then the $\Re$-module $R_d(a, b, c)$ is isomorphic to $V$ if and only if $a$, $b$, $c$ are the roots of
\begin{gather*}
x^2+x+\frac{d(d+2)}{12}=\frac{\operatorname{tr} A}{d+1},\\
x^2+x+\frac{d(d+2)}{12}=\frac{\operatorname{tr} B}{d+1},\\
x^2+x+\frac{d(d+2)}{12}=\frac{\operatorname{tr} C}{d+1},
\end{gather*}
respectively.
\end{Corollary}

We are now ready to prove Theorem~\ref{thm:classification}.

\begin{proof}[Proof of Theorem \ref{thm:classification}] By Theorems~\ref{thm:irrR} and~\ref{thm:iso}, the map $\mathcal R$ is well-defined. By Theorem~\ref{thm:surjective}, the map $\mathcal R$ is onto. Since any element of~$\Re$ has the same trace on the isomorphic finite-dimensional $\Re$-modules, it follows from Lemma \ref{lem:trace} that $\mathcal R$ is one-to-one.
\end{proof}

In Example \ref{exam:R} we showed a five-dimensional irreducible $\Re$-module on which none of~$A$,~$B$,~$C$ is diagonalizable. Note that the $\Re$-module is isomorphic to $R_4\big({-}\frac{1}{2}, -\frac{1}{2}, -\frac{1}{2}\big)$. We finish this paper with the necessary and sufficient conditions for $A$, $B$, $C$ to be diagonalizable on finite-dimensional irreducible $\Re$-modules.

\begin{Theorem}Assume that $\F$ is algebraically closed with $\operatorname{char}\F=0$.
For any $a,b,c\in \F$ and $d\in \N$ satisfying the conditions $(i)$ and $(ii)$ of Theorem~{\rm \ref{thm:irrR}}, the following statements are equivalent:
\begin{enumerate}\itemsep=0pt
\item[$(i)$] $A$ {\rm (}resp. $B${\rm )} {\rm (}resp.~$C${\rm )} is diagonalizable on $R_d(a,b,c)$,
\item[$(ii)$] $a$ {\rm (}resp. $b${\rm )} {\rm (}resp. $c${\rm )} is not in
\begin{gather}\label{diag}
\big\{\tfrac{i-d-1}{2}\,\big|\, i=1,2,\ldots,2d-1\big\}.
\end{gather}
\end{enumerate}
\end{Theorem}
\begin{proof}
With reference to Proposition \ref{prop:Rd}(i), the minimal polynomial of $A$ in $R_d(a,b,c)$ is
\begin{gather*}
\prod_{i=0}^d (x-\theta_i).
\end{gather*}
Thus $A$ is diagonalizable on $R_d(a,b,c)$ if and only if the scalars $\{\theta_i\}_{i=0}^d$ are mutually distinct. A direct calculation yields that the latter holds if and only if $a$ is not in (\ref{diag}). By a similar argument $B$ is diagonalizable on $R_d(a,b,c)$ if and only if $b$ is not in (\ref{diag}).

To derive the condition for $C$ as diagonalizable on $R_d(a,b,c)$, we consider the $(d+1)$-dimensional $\Re$-module $R_d(b,c,a)$. By Theorem~\ref{thm:irrR} the $\Re$-module $R_d(b,c,a)$ is irreducible. By Definition~\ref{defn:Racah} or \cite[Proposition~4.1]{SH:2017-1} there exists a unique $\F$-algebra automorphism $\varrho$ of $\Re$ that sends $A$, $B$, $C$, $D$ to $C$, $A$, $B$, $D$, respectively. Let $R_d(b,c,a)^\varrho$ denote the $\Re$-module obtained by pulling back the $\Re$-module $R_d(b,c,a)$ via $\varrho$. Observe that $C$ is diagonalizable on $R_d(b,c,a)^\varrho$ if and only if $c$ is not in~(\ref{diag}). Using Corollary~\ref{cor:inj} yields that $R_d(b,c,a)^\varrho$ is isomorphic to the $\Re$-module $R_d(a,b,c)$. The result follows.
\end{proof}

\subsection*{Acknowledgements}

The research of the first author is supported by the Ministry of Science and Technology of Taiwan under the project MOST 106-2628-M-008-001-MY4.

\pdfbookmark[1]{References}{ref}
\LastPageEnding

\end{document}